\documentclass[12pt]{amsart}

\usepackage{natbib}
\usepackage{hyperref}

\newcommand{\sgn}{\operatorname{sgn}}

\newcommand{\vague}{\stackrel{v}{\to}}

\newcommand{\eid}{\stackrel{d}{=}}
\newcommand{\one}{{\bf 1}}
\newcommand{\reals}{{\mathbb R}}
\newcommand{\bbr}{\reals}
\newcommand{\vep}{\varepsilon}

\newcommand{\bbrdcomp}{\overline{\bbr^d}}

\newtheorem{theorem}{Theorem}[section]
\newtheorem{lemma}[theorem]{Lemma}

\theoremstyle{definition}

\theoremstyle{remark}
\newtheorem{remark}[theorem]{Remark}

\numberwithin{equation}{section}

\newcommand{\rr}{{\mathbb R}}

\newcommand{\rd}{{\mathbb R^d}}

\newcommand{\ip}[2]{{\langle #1,#2\rangle}}

\begin{document}
\sloppy
\title[Tempered stable laws as random walk limits]{Tempered stable laws as random walk limits}

\author{Arijit Chakrabarty}
\address{Arijit Chakrabarty (corresponding author), Statistics and Mathematics Unit, Indian Statistical Institute\\7 S.J.S. Sansanwal Marg, New Delhi 110016, INDIA\\
Phone number: +91-11-41493918}
\email{arijit@isid.ac.in}

\author{Mark M. Meerschaert}
\address{Mark M. Meerschaert, Department of Statistics and Probability, Michigan State University, East Lansing, MI 48824}
\email{mcubed\@@{}stt.msu.edu}
\urladdr{http://www.stt.msu.edu/$\sim$mcubed/}

%\date{14 May 2010}

\begin{abstract}
Stable laws can be tempered by modifying the L\'evy measure to cool the probability of large jumps.  Tempered stable laws retain their signature power law behavior at infinity, and infinite divisibility.  This paper develops random walk models that converge to a tempered stable law under a triangular array scheme.  Since tempered stable laws and processes are useful in statistical physics, these random walk models can provide a basic physical model for the underlying physical phenomena.
\end{abstract}


%\subjclass{Primary: 60G50, 60F17; Secondary: 60H30, 82C31}

\thanks{AC was partially supported by the Centenary Postdoctoral Fellowship at the Indian Institute of Science. MMM was partially supported by NSF grants DMS-0125486, DMS-0803360, EAR-0823965 and NIH grant R01-EB012079-01.}

\maketitle

\baselineskip=18pt

\section{Introduction}
Tempered stable laws were introduced in physics as a model for turbulent velocity fluctuations (\cite{Koponen:1995, Novikov:1994}).  They have also been used in finance (\cite{carr:geman:madan:yor:2002, carr:geman:madan:yor:2003}) and hydrology (\cite{meerschaert:zhang:baeumer:2008}) as a model of transient anomalous diffusion (\cite{baeumer:meerschaert:2010}).   The general class of tempered stable distributions for random vectors was developed by \cite{Rosinski:2007}.  In short, the L\'evy measure of a stable law is modified in the tail to reduce the probability of large jumps.  Often this is done in such a way that all moments exist, but tempering by a power law of higher order is also useful (\cite{sokolov:chechkin:klafter:2004}).  This paper develops random walk models that converge to a tempered stable law.  Starting with a random walk in the domain of attraction of a stable law, the basic idea is to modify the tails of the jumps to mimic the tempering function of the limit.  A triangular array scheme is essential, since the limit is no longer stable.  The results of this paper are intended to form a useful random walk model for natural processes that are well described by a tempered stable. The main result of this paper is Theorem \ref{t1}, which shows that the weak limit of the row sum of that triangular array is a tempered stable distribution. In Theorem \ref{t4}, we show that the random walk process converges to the L\'evy process generated by the tempered stable distribution in the sense of finite-dimensional distributions.

Section \ref{sec2} gives a brief background of stable distributions and their domains of attraction. In Section \ref{sec3}, we define  tempered stable distributions and the triangular array model. In Section \ref{sec4}, we state and prove the results regarding the convergence of the random walk to the tempered stable distribution.

\section{Stable limits for random walks}\label{sec2}
Recall that a random vector $X$ on $\rd$ is infinitely divisible if and only if its characteristic function $E[e^{i\ip \lambda X}]=e^{\psi(\lambda)}$, where
\begin{equation}\label{LevyRepn1}
\psi(\lambda)=i\ip{a}\lambda-\tfrac12 \ip{\lambda}{Q\lambda}+\int_{x\neq 0}\Bigl(e^{i\ip{\lambda}x}-1-
i\ip{\lambda}{x}\one(\|x\|\le 1)\Bigr)M(dx) ,
\end{equation}
where $a\in \rd$, $Q$ is a nonnegative definite symmetric $d\times d$ matrix with entries in $\rr$, and $M$ is a $\sigma$-finite Borel measure on $\rd\setminus\{0\}$ such that $\int_{x\neq 0}\min\{1,\|x\|^2\}M(dx)<\infty$.
The triple $[a,Q,M]$ is called the L\'evy representation, and it is unique \cite[Theorem 3.1.11]{meerschaert:scheffler:2001}. The measure $M$ is known as the L\'evy measure of  $X$.

A $\bbr^d$ valued random vector $X$ is said to be stable if and only if for all $n\ge1$, there exist $b_n>0$ and $a_n\in\bbr^d$ so that
$X_1+\ldots+X_n\eid b_nX+a_n$,
where $X_1,X_2,\ldots$ are i.i.d. copies of $X$. Clearly, a stable random vector is infinitely divisible. It is well known that given a stable random vector, either it is Gaussian, or the Gaussian part is completely absent, {\it i.e.}, in the L\'evy representation \eqref{LevyRepn1}, $Q=0$. In this paper, ``stable random vector'' will refer to the latter case, {\it i.e.}, non-Gaussian stable random vectors. For such a random vector $X$, $P(\|X\|>\cdot)$ regularly varies with index $-\alpha$ for some $0<\alpha<2$. Sometimes, $X$ is also referred to as an $\alpha$-stable random vector. The L\'evy representation of the random vector $X$ is $[a,0,M_0]$ for some $a\in\rd$ where
$M_0(dr,ds)=r^{-\alpha-1}dr\sigma(ds)$,
and $\sigma$ is a finite non-zero Borel measure on the unit sphere $S^{d-1}=\{x\in\rd:\|x\|=1\}$,
see for example \cite[Theorem 7.3.16]{meerschaert:scheffler:2001}. The measure $\sigma$ is known as the spectral measure of $X$. For more details on stable distributions, the reader is referred to \cite{samorodnitsky:taqqu:1994}.

The domain of attraction of an $\alpha$-stable random vector $X$ consists of $\bbr^d$ valued random vectors $H$ such that there exist a sequence of positive numbers $(b_n)$ and a sequence $(a_n)$ in $\bbr^d$ satisfying
\begin{equation}\label{doadef}
b_n^{-1}\left(H_1+\cdots+H_n\right)-a_n\Rightarrow X
\end{equation}
as $n\to\infty$, where $\Rightarrow$ denotes weak convergence and $H_1,H_2,\ldots$ are i.i.d. copies of $H$.   A necessary and sufficient condition for \eqref{doadef} is that $V(r)=P(\|H\|>r)$ varies regularly with index $-\alpha$, and
\begin{equation}\label{TailBalDOAd}
P\bigl(\tfrac {H}{\|H\|}\in D\bigl|
\|H\|>r\bigr)=\frac{P(\|H\|>r,\frac{H}{\|H\|} \in D)}{V(r)}\to
\frac{\sigma(D)}{\sigma(S^{d-1})}
\end{equation}
as $r\to\infty$ for all Borel subsets $D$ of $S^{d-1}$
such that $\sigma(\partial
D)=0$, see for example \cite[Theorem 7.3.16]{meerschaert:scheffler:2001}.  When $\alpha>1$, $m=E(H)$ exists, and we can center to zero expectation in \eqref{doadef} by setting $a_n=nb_n^{-1} m$.  Then the limit $X$ also has zero mean, and its log-characteristic function
\begin{equation}\label{LevyRepn2}
\psi(\lambda)=\int_{x\neq 0}\Bigl(e^{i\ip{\lambda}x}-1-
i\ip{\lambda}{x}\Bigr)M(dx) .
\end{equation}
When $0<\alpha<1$, no centering is required:  Set $a_n=0$ in \eqref{doadef}; then $X$ is a centered stable law with log-characteristic function
\begin{equation}\label{LevyRepn3}
\psi(\lambda)=\int_{x\neq 0}\Bigl(e^{i\ip{\lambda}x}-1\Bigr)M(dx) .
\end{equation}
See for example \cite[Theorem 8.2.16]{meerschaert:scheffler:2001}.

Suppose that $X$ is a stable random vector.
Let $\{X(t)\}$ denote the L\'evy process associated with $X$, so that $X(0)=0$ almost surely, $X(t)$ has stationary, independent increments, and $X(1)=X$ in distribution.  Suppose that \eqref{doadef} holds with $a_n=0$, let $b(c)=b_{\lceil c\rceil}$, and $S(c)=\sum_{j=1}^{\lceil c\rceil}H_j$ for $c\ge0$.  Then, as $c\to\infty$, we also have process convergence $\{b(c)^{-1}S(ct)\}_{t\ge0}\Rightarrow \{X(t)\}_{t\ge0}$ in the sense of finite dimensional distributions \cite[Example 11.2.18]{meerschaert:scheffler:2001} as well as convergence in the Skorokhod space $D([0,\infty),\rd)$ of right continuous functions with left hand limits, in the $J_1$ topology \cite[Theorem 4.1]{meerschaert:scheffler:2004}.  The random vectors $X(t)$ have smooth density functions $P(x,t)$ that solve a fractional differential equation
$\frac{\partial}{\partial t} P(x,t)={\mathcal D}\nabla^\alpha_\sigma P(x,t)$
for anomalous diffusion (\cite{meerschaert:benson:baeumer:1999}).  The fractional derivative operator $\nabla^\alpha_M f(x)$ is defined, for suitable functions $f(x)$ with Fourier transform $\hat f(\lambda)=\int e^{i\ip \lambda x} f(x)\,dx$, as the inverse Fourier transform of $\int_{\|s\|=1} (-i\ip \lambda s)^\alpha \sigma(ds) \hat f(\lambda)$, and ${\mathcal D}>0$ is a positive constant that depends on choice of the norming sequence $b(c)$. The random walk $S_n$ provides a physical model for particle jumps, whose ensemble behavior is approximated by the stable density functions $P(x,t)$.  For example, the random walk can be simulated to solve the fractional diffusion equation, a numerical method known as particle tracking (\cite{zhang:benson:meerschaert:labolle:scheffler:2006}).  The purpose of this paper is to develop analogous random walk models for tempered stables.

\section{Tempered random walks}\label{sec3}
This section develops random walk models that converge to a tempered stable, using a triangular array scheme.  As in \cite{Rosinski:2007}, we define a $d$-dimensional proper tempered $\alpha$-stable random vector to be an infinitely divisible random vector with L\'evy representation $[a,0,M]$ with
\begin{equation}\label{defM}
M(dr,ds)=r^{-\alpha-1}q(r,s)dr\sigma(ds)
\end{equation}
for $r>0$ and $s\in S^{d-1}$. Here $\alpha\in(0,2)$, $\sigma$ is a finite Borel measure on the unit sphere $S^{d-1}=\{x\in\rd:\|x\|=1\}$, and $q:(0,\infty)\times S^{d-1}$ is a Borel measurable function such that for all $s\in S^{d-1}$, $q(\cdot,s)$ is non-increasing, $q(0+,s)=\alpha$ and $q(\infty,s)=0$. We also assume that $q$ is continuous in the second variable, {\it i.e.}, $q(r,\cdot)$ is continuous for all $r>0$.
In \cite{Rosinski:2007}, the assumption is that $q(0+,s)=1$. However, a simple reparametrization yields $q(0+,s)=\alpha$.  It is also assumed in \cite{Rosinski:2007} that $q(\cdot,s)$ is completely monotone, but we do not need that assumption in this paper.  Note that tempered stable random vectors are full dimensional, since the L\'evy measure is not concentrated on any lower dimensional subspace \cite[Proposition 3.1.20]{meerschaert:scheffler:2001}.
Let $H$ be a random vector in the domain of attraction of $X$ such that (without loss of generality) $P(H=0)=0$.
%It is immediate that \eqref{eq1} is not a loss of generality:  The random vector $X$ itself satisfies this condition, for example. Also, for any random vector $Y$ in the domain of attraction of $X$, the slightly modified random variable $H=Y+x\one(Y=0)$ (here $x\in\bbr^d\setminus\{0\}$) satisfies \eqref{eq1} and is in the domain of attraction of $X$.  Also, $Y-y$ satisfies this condition for any $y\in\rd$ such that $P(Y=y)=0$.
We will define a random walk that approximates the tempered stable using a conditional tempering of the jumps.  Define a function $\pi:(0,\infty)\times S^{d-1}\to\bbr$ by
\begin{equation}\label{pidef}
\pi(u,s)=u^\alpha\int_u^\infty r^{-\alpha-1}q(r,s)dr\,.
\end{equation}
From the fact that the function $q$ is bounded above by $\alpha$, it is immediate that the integral on the right hand side is finite. Clearly,
\begin{equation*}\begin{split}\label{eq3}
\frac{\partial\pi(u,s)}{\partial u}&=\alpha u^{\alpha-1}\int_u^\infty r^{-\alpha-1}q(r,s)dr-u^{-1}q(u,s)\\
&=\alpha u^{\alpha-1}\int_u^\infty r^{-\alpha-1}\left\{q(r,s)-q(u,s)\right\}dr\nonumber\le 0 ,
\end{split}\end{equation*}
the inequality following from the fact that $q(\cdot,s)$ is non-increasing. Thus $\pi(\cdot,s)$ is also non-increasing. A simple application of the L'H\^{o}pital's rule yields that $\pi(0+,s)=1$ and $\pi(\infty,s)=0$.

Define a family of probability measures on $(0,\infty)$ by
$\Pi(du,s)=-\frac{\partial}{\partial u}\pi(u,s)du$.
Since $q$ is a measurable function, for every Borel set $A\subset(0,\infty)$, $\Pi(A,\cdot)$ is a measurable function from $S^{d-1}$ to $\bbr$. Hence there exists a random variable $T>0$ whose conditional distribution given $H$ is $\Pi\left(\cdot,H/{\|H\|}\right)$.

Now we construct the tempered random walk.  For $t>0$, define
\begin{equation}\label{Htdef}
H^t=\frac{H}{\|H\|} (\|H\| \wedge t T)\,.
\end{equation}
Let $\{(H_j:T_j):j\ge1\}$ be i.i.d.\ copies of $(H,T)$.  Suppose $v_n\to\infty$ is a sequence of positive numbers. Define $Y_{nj}:=\left(\|H_j\|\wedge v_nT_j\right){H_j}/{\|H_j\|}$ for $n,j\ge1$, and let
\begin{equation}\label{TRWdef}
S_n(k):= Y_{n1}+\cdots+Y_{nk} .
\end{equation}
The next section shows that, for suitably chosen truncation thresholds $v_n$, the random walk \eqref{TRWdef} is asymptotically tempered stable.

\section{Limits of tempered random walks}\label{sec4}
The results in this section show that for suitable truncation thresholds $v_n$ and centering, the random walk \eqref{TRWdef} converges to a tempered stable.  The form of the centering is then related to the case with no tempering.  We begin with a few technical results.  Recall that for sigma-finite Borel measures $\mu_n,\mu$ on $\Gamma=\bbrdcomp\setminus\{0\}$, $\mu_n\vague\mu$ (vague convergence) means that $\mu_n(D)\to\mu(D)$ for Borel sets $D\subset\Gamma$ that are bounded away from zero, for which $\mu(\partial D)=0$.

\begin{lemma}\label{l1} Suppose that $H^t$ is defined by \eqref{Htdef}. Then
$$
\frac{P(t^{-1}H^t\in\cdot)}{P(\|H\|>t)}\vague\frac1{\sigma(S^{d-1})}M(\cdot)
$$
on $\Gamma$ as $t\to\infty$, where $M$ is as in \eqref{defM}.
\end{lemma}

For the proof, we shall need the following result from weak convergence. This result is similar to Theorem 1.3.4 in \cite{vandervaart:wellner:1996}, which states the corresponding result for probability measures.  Although the result is well known, we include the proof here for completeness, since we could not locate a suitable reference.

\begin{lemma}\label{l2} Suppose $(\mu_n)$ is a sequence of measures on some metric space $\mathcal S$ equipped with the Borel sigma-field, converging weakly to some finite measure $\mu$. Then, for all bounded non-negative upper semicontinuous functions $f$, we have
$
\limsup_{n\to\infty}\int fd\mu_n\le\int fd\mu\,.
$
\end{lemma}

\begin{proof} Since $f$ is bounded and non-negative, we can assume without loss of generality that $0\le f<1$. For $k\ge1$, denote $f_k:=\frac1k\sum_{i=1}^k\one\left(f^{-1}\left[\frac{i-1}k,1\right)\right)$.
It is easy to see that $f\le f_k\le f+\frac1k$.
Thus, for fixed $k\ge1$,
\begin{eqnarray*}
\limsup_{n\to\infty}\int fd\mu_n&\le&\limsup_{n\to\infty}\int f_kd\mu_n
%&=&\limsup_{n\to\infty}\frac1k\sum_{i=1}^k\mu_n\left(f^{-1}\left[\frac{i-1}k,1\right)\right)\\
\le\frac1k\sum_{i=1}^k\limsup_{n\to\infty}\mu_n\left(f^{-1}\left[\frac{i-1}k,1\right)\right)\\
&\le&\frac1k\sum_{i=1}^k\mu\left(f^{-1}\left[\frac{i-1}k,1\right)\right)
%&=&\int f_kd\mu\\
\le\int fd\mu+\frac1k\mu({\mathcal S})\,,
\end{eqnarray*}
by the Portmanteau Theorem (Theorem 2.1, \cite{billingsley:1968})
and the observation that $f^{-1}\left[\frac{i-1}k,1\right)$
is a closed set because $f$ is upper semicontinuous. Since $\mu$ is a finite measure, this completes the proof.
\end{proof}

\begin{proof}[Proof of Lemma \ref{l1}]Note that if $q(r,\cdot)$ is continuous for all $r>0$, then the same is true for $\pi(u,\cdot)$ for all $u>0$. To see this, suppose that the former holds. Fix $u>0$ and $s_n,s\in S^{d-1}$ such that $s_n\to s$. By the assumption on $q$ and the dominated convergence theorem, it follows that
$
\lim_{n\to\infty}\int_u^\infty r^{-\alpha-1}q(r,s_n)dr=\int_u^\infty r^{-\alpha-1}q(r,s)dr\,.
$
Then $\pi(u,s_n)\to\pi(u,s)$, so $\pi(u,\cdot)$ is continuous for every $u>0$. A similar argument shows that $\frac{\partial\pi(u,\cdot)}{\partial u}$ is continuous for every $u>0$.

For the proof, we shall use the fact that as $t\to\infty$,
\begin{equation}\label{l1.eq3}
P(\|H\|>t)^{-1}P\left[t^{-1}\|H\|\in dr,\frac H{\|H\|}\in ds\right]\vague\frac\alpha{\sigma(S^{d-1})}r^{-\alpha-1}dr\sigma(ds)
\end{equation}
 which is a restatement of Theorem 8.2.18 in \cite{meerschaert:scheffler:2001}.
It suffices to show that for every closed set $A\subset\bbr^d\setminus\{0\}$,
\begin{equation}\label{l1.eq1}
\limsup_{t\to\infty}{P(t^{-1}H^t\in A)}{P(\|H\|>t)^{-1}}\le{\sigma(S^{d-1})^{-1}}M(A)\,,
\end{equation}
and that for every $\vep>0$,
\begin{equation}\label{l1.eq2}
\lim_{t\to\infty}{P(t^{-1}H^t\in B_\vep^c)}{P(\|H\|>t)^{-1}}={\sigma(S^{d-1})^{-1}}M(B_\vep^c)\,,
\end{equation}
where $B_r$ is the closed ball of radius $r$ centered at origin:
$
B_r:=\{x\in\bbr^d:\|x\|\le r\}\,.
$
Fix a closed set $A\subset\bbr^d\setminus\{0\}$ and note that
$$
P(t^{-1}H^t\in A)
=\int_0^\infty\int_{S^{d-1}}\int_0^\infty\one_A((r\wedge u)s)\Pi(du,s)\ P\big(t^{-1}\|H\|\in dr,\frac H{\|H\|}\in ds\big)
$$
Fix sequences $r_n$ and $s_n$ so that $r_n>0$, $r_n\to r>0$, $s_n\in S^{d-1}$ and $s_n\to s$. Then
\begin{eqnarray*}
&&\limsup_{n\to\infty}\int_0^\infty\one_A((r_n\wedge u)s_n)\Pi(du,s_n)\\
%&=&\limsup_{n\to\infty}\int_0^\infty\left\{1-\one_{A^c}((r_n\wedge u)s_n)\right\}\Pi(du,s_n)\\
%&=&1-\liminf_{n\to\infty}\int_0^\infty\one_{A^c}((r_n\wedge u)s_n)\Pi(du,s_n)\\
%&=&1-\liminf_{n\to\infty}\int_0^\infty\one_{A^c}((r_n\wedge u)s_n)\left(-\frac{\partial\pi(u,s_n)}{\partial u}\right)du\\
&\le&1-\int_0^\infty\liminf_{n\to\infty}\one_{A^c}((r_n\wedge u)s_n)\left(-\frac{\partial\pi(u,s_n)}{\partial u}\right)du\\
&=&1-\int_0^\infty\left(-\frac{\partial\pi(u,s)}{\partial u}\right)\liminf_{n\to\infty}\one_{A^c}((r_n\wedge u)s_n)du
%1-\int_0^\infty\left(-\frac{\partial\pi(u,s)}{\partial u}\right)\one_{A^c}((r\wedge u)s)du\\&=&
\le\int_0^\infty\one_A((r\wedge u)s)\Pi(du,s)
\end{eqnarray*}
by Fatou's Lemma, continuity of $\frac{\partial}{\partial u}\pi(u,\cdot)$, and the fact that $A^c$ is open. 
Then 
$
(r,s)\mapsto\int_0^\infty\one_A((r\wedge u)s)\Pi(du,s)
$
is upper semicontinuous. 
From \eqref{l1.eq3}, it follows that for all $\vep>0$, the restriction of
$
P(\|H\|>t)^{-1}P\left[t^{-1}\|H\|\in dr,H/{\|H\|}\in ds\right]
$
to $B_\vep^c$ converges weakly to that of
$
\alpha r^{-\alpha-1}dr\sigma(ds)/{\sigma(S^{d-1})}.
$
Thus, by Lemma \ref{l2} and the fact that $A$ is bounded away from zero, it follows that
\begin{equation*}\begin{split}
\limsup_{t\to\infty}&\frac{P(t^{-1}H^t\in A)}{P(\|H\|>t)}
\le\int_0^\infty\int_{S^{d-1}}\int_0^\infty\one_A((r\wedge u)s)\Pi(du,s)\frac\alpha{\sigma(S^{d-1})}r^{-\alpha-1}dr\sigma(ds)\\
&=\int_0^\infty\int_{S^{d-1}}\left\{\int_0^r\one_A(us)\Pi(du,s)\right\}
\frac\alpha{\sigma(S^{d-1})}r^{-\alpha-1}dr\sigma(ds)\\
&+\int_0^\infty\int_{S^{d-1}}\left\{\int_r^\infty\one_A(rs)
\Pi(du,s)\right\}\frac\alpha{\sigma(S^{d-1})}r^{-\alpha-1}dr\sigma(ds)
=:I_1+I_2\,.
\end{split}\end{equation*}
Note that Lemma \ref{l2} applies since $\Pi(du,s)$ is a probability measure for each $s$.

A change of the order of integration yields that
$$
I_1=\frac1{\sigma(S^{d-1})}\int_0^\infty\int_{S^{d-1}}\one_A(us)u^{-\alpha}\left(-\frac{\partial\pi(u,s)}{\partial u}\right)\sigma(ds)du\,.
$$
It is immediate that
$
I_2={\sigma(S^{d-1})^{-1}}\int_0^\infty\int_{S^{d-1}}\one_A(rs)\alpha\pi(r,s)r^{-\alpha-1}dr\sigma(ds)\,.
$
Thus,
\begin{equation}\begin{split}\label{l1.eq4}
&\limsup_{t\to\infty}\frac{P(t^{-1}H^t\in A)}{P(\|H\|>t)}\\
&\le\frac1{\sigma(S^{d-1})}\int_0^\infty\int_{S^{d-1}}\one_A(rs)\left\{\alpha\pi(r,s)-r\frac{\partial\pi(r,s)}{\partial r}\right\}r^{-\alpha-1}dr\sigma(ds)\,.
\end{split}\end{equation}
From \eqref{eq3}, it follows that
$
\alpha\pi(r,s)-r\frac{\partial}{\partial r}\pi(r,s)=q(r,s)\,.
$
Plugging this in \eqref{l1.eq4} yields that
\begin{equation*}\begin{split}
\limsup_{t\to\infty}\frac{P(t^{-1}H^t\in A)}{P(\|H\|>t)}
&\le\frac1{\sigma(S^{d-1})}\int_0^\infty\int_{S^{d-1}}\one_A(rs)q(r,s)r^{-\alpha-1}dr\sigma(ds)=\frac{M(A)}{\sigma(S^{d-1})}\,,
\end{split}\end{equation*}
thus showing \eqref{l1.eq1}.  For \eqref{l1.eq2}, note that as $t\to\infty$,
\begin{equation*}\begin{split}
P(t^{-1}H^t\in B_\vep^c)
%&=&P(\|H\|>t\vep,T>\vep)\\
&=\int_\vep^\infty\int_{S^{d-1}}\pi(\vep,s)P\left(t^{-1}\|H\|\in dr,\frac H{\|H\|}\in ds\right)\\
&\sim P(\|H\|>t)\int_\vep^\infty\int_{S^{d-1}}\pi(\vep,s)\frac\alpha{\sigma(S^{d-1})}r^{-\alpha-1}dr\sigma(ds)\\
&=P(\|H\|>t)\frac1{\sigma(S^{d-1})}\int_{S^{d-1}}\vep^{-\alpha}\pi(\vep,s)\sigma(ds)
%&=&P(\|H\|>t)\frac1{\sigma(S^{d-1})}\int_{S^{d-1}}\int_\vep^\infty r^{-\alpha-1}q(r,s)dr\sigma(ds)\\
=\frac{P(\|H\|>t)}{\sigma(S^{d-1})}M(B_\vep^c)\,,
\end{split}\end{equation*}
by \eqref{l1.eq3}, the fact that $\pi(\vep,\cdot)$ is continuous, and the definition of $\pi$. 
\end{proof}

The following theorem is the main result of this paper.  It shows that the tempered random walk \eqref{TRWdef} converges weakly to a tempered stable, for suitable tempering constants $v_n$ and suitable centering vectors $a_n$.

\begin{theorem}\label{t1} For $n\ge1$ let
$
b_n:=\inf\left\{x:P(\|H\|>x)\le n^{-1}\right\}\,.
$
If the sequence $(v_n)$ satisfies
\begin{equation}\label{t1.eq3}
\lim_{n\to\infty}v_n^{-1}{b_n}=\sigma(S^{d-1})^{1/\alpha}\,,
\end{equation}
then,
\begin{equation}\label{t1.claim}
v_n^{-1}S_n(n)-a_n\Rightarrow\rho
\end{equation}
where $\rho$ is an infinitely divisible probability measure on $\bbr^d$ with L\'evy representation $[0,0,M]$,
and $(a_n)$ is defined by
\begin{equation}\label{an:defn}
a_n:=n\int_{\{\|x\|<1\}}xP(v_n^{-1}Y_{n1}\in dx)\,.
\end{equation}
\end{theorem}

\begin{proof} Note that
$nP(\|H\|>b_n)\to 1$ \cite[p.\ 24]{resnick:2007}. 
Since $P(\|H\|>\cdot)$ varies regularly with index $-\alpha$, \eqref{t1.eq3} implies
$
{P(\|H\|>v_n)}/{P(\|H\|>b_n)}\to \sigma(S^{d-1})\,.
$
Then
\begin{equation}\label{t1.eq5}
\lim_{n\to\infty}nP(\|H\|>v_n)=\sigma(S^{d-1})\,.
\end{equation}
An appeal to Lemma \ref{l1} shows that
\begin{equation}\label{t1.eq6}
nP(v_n^{-1}Y_{n1}\in\cdot)\vague M(\cdot) ,
\end{equation}
and then it suffices to check \cite[Theorem 3.2.2]{meerschaert:scheffler:2001}:
\begin{equation}\label{t1.eq4}
\lim_{\delta\downarrow0}\limsup_{n\to\infty}nv_n^{-2}E\left\|Y_{n1}\one(\|Y_{n1}\|\le v_n\delta)\right\|^2=0\,.
\end{equation}
For this, note that
$\|Y_{n1}\|^2\one(\|Y_{n1}\|\le v_n\delta)
\le\|H_1\|^2\one(\|H_1\|\le v_n\delta)+v_n^2T_1^2\one(v_nT_1\le v_n\delta,\|H_1\|>v_n\delta)
\le\|H_1\|^2\one(\|H_1\|\le v_n\delta)+v_n^2\delta^2\one(\|H_1\|>v_n\delta)$.
Since $P(\|H\|>\cdot)$ is regularly varying with index $-\alpha$ and $\alpha<2$, by Karamata's theorem \cite[Theorem 2.1]{resnick:2007} it follows that
$E\left[\|H\|^2\one(\|H\|\le v_n\delta)\right]\sim(v_n\delta)^2P(\|H\|>v_n\delta)\alpha/(2-\alpha)
\sim\delta^{2-\alpha}v_n^{2}P(\|H\|>v_n)\alpha/(2-\alpha)$ as $n\to\infty$.
Using the regular variation of $P(\|H\|>\cdot)$ once again, it is immediate that
$
E\left[v_n^2\delta^2\one(\|H\|>v_n\delta)\right]=v_n^2\delta^2P(\|H\|>v_n\delta)\sim\delta^{2-\alpha}v_n^2P(\|H\|>v_n)\,.
$
To complete the proof, use \eqref{t1.eq5} to obtain $C<\infty$ such that for all $\delta>0$,
\begin{equation}\label{t1.eq7}
\limsup_{n\to\infty}nv_n^{-2}E\left\|Y_{n1}\one(\|Y_{n1}\|\le v_n\delta)\right\|^2\le C\delta^{2-\alpha}\,.
\end{equation}
\end{proof}

The next two results show that the centering constants in \eqref{t1.claim} can be chosen in the same way as for \eqref{doadef} when $\alpha\neq1$.  We say that a tempered stable law with index $0<\alpha<1$ is centered if its log-characteristic function can be written in the form \eqref{LevyRepn3} where $M$ is given by \eqref{defM}.

\begin{theorem}\label{t2} If $0<\alpha<1$ and $v_n$ satisfies \eqref{t1.eq3}, then
$
v_n^{-1}S_n(n)\Rightarrow\rho_1
$
where $\rho_1$ is centered tempered stable.
\end{theorem}

\begin{proof} In view of Theorem \ref{t1} and \eqref{LevyRepn1}, it suffices to show that, if $a_n$ is defined by \eqref{an:defn}, then
$
a_n\to \int_{\{\|x\|<1\}}xM(dx)\,.
$
Fix $0<\vep<1$ and note that
$$
a_n=n\int_{\{\vep<\|x\|<1\}}xP(v_n^{-1}Y_{n1}\in dx)+n\int_{\{\|x\|\le\vep\}}xP(v_n^{-1}Y_{n1}\in dx)
=:I_1+I_2\,.
$$
Clearly, by \eqref{t1.eq6},
$
\lim_{n\to\infty}I_1=\int_{\{\vep<\|x\|<1\}}xM(dx)\,.
$
Thus, it suffices to show
\begin{equation}\label{t2.eq1}
\lim_{\vep\downarrow0}\limsup_{n\to\infty}\|I_2\|=0\,.
\end{equation}
Note
$\|I_2\|\le nv_n^{-1}E[\|Y_{n1}\|\one(\|Y_{n1}\|\le v_n\vep)]
\le nv_n^{-1}[E(\|H\|\one(\|H\|\le v_n\vep))+v_n\vep P(\|H\|>v_n\vep)]$.
Since $\alpha<1$, Karamata along with regular variation yields
$E\left(\|H\|\one(\|H\|\le v_n\vep)\right)\sim v_n\vep^{1-\alpha}P(\|H\|>v_n)\alpha/(1-\alpha)$.
Then \eqref{t2.eq1} follows, using \eqref{t1.eq5} and regular variation.
\end{proof}

\begin{theorem}\label{t3} Suppose $\alpha>1$ and that for some $\beta>\alpha$,
\begin{equation}\label{t3.hypo}
\limsup_{u\downarrow0}\sup_{s\in S^{d-1}}{u^{1-\beta}}[\alpha-q(u,s)]<\infty
\end{equation}
where $v_n$ satisfies \eqref{t1.eq3}.  Then,
$
v_n^{-1}[S_n(n)- nE(H)]\Rightarrow\rho_2
$
where $\rho_2$ is an infinitely divisible law with no Gaussian component, L\'evy measure $M$ and mean
\begin{equation}\label{tsmean}
m=-\alpha\int_0^\infty\int_{S^{d-1}}\left\{\int_0^r(r-u)s\Pi(du,s)\right\}r^{-\alpha-1}dr\sigma(ds) .
\end{equation}
\end{theorem}

\begin{proof}
Let
\begin{equation}\label{t3.claim}
\theta:=\alpha\int_0^\infty\int_{S^{d-1}}\left\{\int_0^r(r-u)s\Pi(du,s)\right\}r^{-\alpha-1}dr\sigma(ds)+\int_{\{\|x\|\ge1\}}xM(dx)\,.
\end{equation}
We start with showing that the integrals on the right hand side of \eqref{t3.claim} are well defined.  Let
$
g(r,s):=\int_0^r(r-u)s\Pi(du,s)\,.
$
It is easy to see that
$
\|g(r,s)\|\le\int_0^r(r-u)\Pi(du,s)=:\bar g(r,s)\,.
$
Clearly,
\begin{equation}\label{t3.eq9}
\bar g(r,s)=r[1-\pi(r,s)]+\int_0^ru\frac{\partial\pi(u,s)}{\partial u}du\,,
\end{equation}
and hence,
$\frac{\partial}{\partial r}\bar g(r,s)=1-\pi(r,s)
=r^\alpha\int_r^\infty u^{-\alpha-2+\beta}u^{1-\beta}[\alpha-q(u,s)]du
=r^\alpha\int_r^1 u^{-\alpha-2+\beta}u^{1-\beta}[\alpha-q(u,s)]du+O(r^\alpha)$.  
Clearly \eqref{t3.hypo} holds with $\beta$ replaced by $\beta\wedge2$. Thus, without loss of generality, we can assume that $\beta\le2$.
Define
$K=\sup\{u^{1-\beta}[\alpha-q(u,s)]:s\in S^{d-1},0<u\le1\}$.
By hypothesis, $K<\infty$. Thus,
$$
r^\alpha\int_r^1 u^{-\alpha-2+\beta}\frac{\alpha-q(u,s)}{u^{\beta-1}}du
\le Kr^\alpha\int_r^1 u^{-\alpha-2+\beta}du\le K'r^{\beta-1}\,,
$$
where $K'=K/(\alpha+1-\beta)>0$ since $\beta\le2$ and $\alpha>1$. Thus, as $r\downarrow0$,
$
\frac{\partial}{\partial r}\bar g(r,s)=O(r^{\beta-1})
$
uniformly in $s$, and hence for some $C<\infty$,
\begin{equation}\label{t3.eq7}
\bar g(r,s)\le Cr^\beta,\,r\le1,s\in S^{d-1}\,.
\end{equation}
It follows that
$
\int_0^1\int_{S^{d-1}}\bar g(r,s)r^{-\alpha-1}dr\sigma(ds)<\infty\,.
$
It is easy to see that $\bar g(r,s)\le r$. Since $\alpha>1$,
$
\int_1^\infty\int_{S^{d-1}}\bar g(r,s)r^{-\alpha-1}dr\sigma(ds)<\infty\,.
$
Thus, the first integral in \eqref{t3.claim} is well defined. Since $\alpha>1$, it is easy to check that the second integral is also well defined.
Then it follows, using Theorem 3.1.14 and Remark 3.1.15 in \cite{meerschaert:scheffler:2001}, that any tempered stable law with index $\alpha>1$ has a finite mean.

Next we want to show that
\begin{equation}\label{t3.eq2}
\lim_{n\to\infty}\left[\frac n{v_n}E(H)-a_n\right]=\theta\,.
\end{equation}
Write
$
nv_n^{-1}E(H)-a_n=nE\left[v_n^{-1}(H_1-Y_{n1})\right]+nv_n^{-1}E\left[Y_{n1}\one(\|Y_{n1}\|\ge v_n)\right]=I_1+I_2\,.
$
Fix $1<N<\infty$ and write
\begin{eqnarray*}
I_2
%&=&n\int_{\{\|x\|\ge1\}}xP\left(v_n^{-1}Y_{n1}\in dx\right)\\
&=&n\int_{\{1\le\|x\|<N\}}xP\left(v_n^{-1}Y_{n1}\in dx\right)+n\int_{\{\|x\|\ge N\}}xP\left(v_n^{-1}Y_{n1}\in dx\right):=I_{21}+I_{22}\,.
\end{eqnarray*}
By \eqref{t1.eq6}, it follows that
\begin{equation}\label{t3.eq1}
\lim_{n\to\infty}I_{21}=\int_{\{1\le\|x\|<N\}}xM(dx)\,.
\end{equation}
Using Karamata's Theorem we get
$\|I_{22}\|
%=\frac n{v_n}\left\|E\left[Y_{n1}\one(\|Y_{n1}\|\ge v_nN)\right]\right\|\\
\le nv_n^{-1}E\left[\|Y_{n1}\|\one(\|Y_{n1}\|\ge v_nN)\right]
\le nv_n^{-1}E\left[\|H\|\one(\|H\|\ge v_nN)\right]
\sim N^{1-\alpha}nP(\|H\|>v_n) \alpha/(\alpha-1)$
as $n\to\infty$. This, in view of \eqref{t1.eq5} show that
$
\lim_{N\to\infty}\limsup_{n\to\infty}\|I_{22}\|=0\,.
$
In conjunction with \eqref{t3.eq1}, this shows that
\begin{equation}\label{t3.eq3}
\lim_{n\to\infty}I_2=\int_{\{\|x\|\ge1\}}xM(dx)\,.
\end{equation}

It remains to show that
\begin{equation}\label{t3.eq4}
\lim_{n\to\infty}I_1=\alpha\int_0^\infty\int_{S^{d-1}}\left\{\int_0^r(r-u)s\Pi(du,s)\right\}r^{-\alpha-1}dr\sigma(ds)\,.
\end{equation}
To that end, fix $0<\vep<1<N<\infty$ and note that
\begin{eqnarray*}
I_1&=&nE\left[\frac H{\|H\|}\left(v_n^{-1}{\|H\|}-T\right)\one\left(v_n^{-1}{\|H\|}>T\right)\right]\\
%&=&n\int_0^\infty\int_{S^{d-1}}g(r,s)P\left(v_n^{-1}{\|H\|}\in dr,\frac H{\|H\|}\in ds\right)\\
&=&n\int_0^\vep\int_{S^{d-1}}g(r,s)P\left(v_n^{-1}{\|H\|}\in dr,\frac H{\|H\|}\in ds\right)\\
&&+n\int_\vep^N\int_{S^{d-1}}g(r,s)P\left(v_n^{-1}{\|H\|}\in dr,\frac H{\|H\|}\in ds\right)\\
&&+n\int_N^\infty\int_{S^{d-1}}g(r,s)P\left(v_n^{-1}{\|H\|}\in dr,\frac H{\|H\|}\in ds\right)=:I_{11}+I_{12}+I_{13}\,.
\end{eqnarray*}

We shall now show that $g$ is jointly continuous. Clearly,
$
g(r,s)=\bar g(r,s)s\,.
$
Thus, it suffices to show that $\bar g$ is jointly continuous. Since $q$ is assumed to be continuous in the second variable, an appeal to the dominated convergence theorem shows that $\pi$ is jointly continuous. By \eqref{eq3}, it follows that $\frac{\partial}{\partial u}\pi(u,\cdot)$ is continuous for every $u>0$. In view of \eqref{t3.eq9}, it suffices to show that the function
$
(r,s)\mapsto\int_0^ru\frac{\partial}{\partial u}\pi(u,s)du
$
is jointly continuous. For that, fix a sequence $r_n\to r$ and $s_n\to s$. Note that
\begin{eqnarray*}
\int_0^{r_n}u\frac{\partial \pi(u,s_n)}{\partial u}du
=\int_0^{r_n}u\frac{\partial\pi(u,s)}{\partial u}du+\int_0^{r_n}u\left[\frac{\partial\pi(u,s_n)}{\partial u}-\frac{\partial\pi(u,s)}{\partial u}\right]du
=:J_1+J_2\,.
\end{eqnarray*}
Clearly, as $n\to\infty$,
$J_1\to\int_0^ru\frac{\partial}{\partial u}\pi(u,s)du\,.$
Let $R=\sup_{n\ge1}r_n$ and note that,
\begin{eqnarray*}
|J_2|&\le&\int_0^Ru\left|\frac{\partial\pi(u,s_n)}{\partial u}-\frac{\partial\pi(u,s)}{\partial u}\right|du
\le R\int_0^\infty\left|\frac{\partial\pi(u,s_n)}{\partial u}-\frac{\partial\pi(u,s)}{\partial u}\right|du\\
&=&R\left[2\int_0^\infty\left\{\frac{\partial\pi(u,s_n)}{\partial u}\vee\frac{\partial\pi(u,s)}{\partial u}\right\}du+2\right]\,,
\end{eqnarray*}
the second equality following from the identity $|a-b|=2(a\vee b)-(a+b)$. Since,
$
\left|\frac{\partial}{\partial u}\pi(u,s_n)\vee\frac{\partial}{\partial u}\pi(u,s)\right|\le-\frac{\partial}{\partial u}\pi(u,s)\,,
$
an appeal to the dominated convergence theorem along with the fact that $\frac{\partial}{\partial u}\pi(u,\cdot)$ is continuous shows that
$
\lim_{n\to\infty}\int_0^\infty\left\{\frac{\partial}{\partial u}\pi(u,s_n)\vee\frac{\partial}{\partial u}\pi(u,s)\right\}du=-1\,,
$
which in turn shows that $J_2\to 0$ as $n\to\infty$.
This shows that $g$ is jointly continuous.

By \eqref{l1.eq3}, \eqref{t1.eq5} and the fact that $g$ is jointly continuous, it follows that
$
\lim_{n\to\infty}I_{12}=\alpha\int_\vep^N\int_{S^{d-1}}g(r,s)r^{-\alpha-1}dr\sigma(ds)\,.
$
Note that
\begin{eqnarray*}
\|I_{11}\|&\le&n\int_0^\vep\int_{S^{d-1}}\bar g(r,s)P\left(v_n^{-1}{\|H\|}\in dr,\frac H{\|H\|}\in ds\right)\\
&\le&Cn\int_0^\vep\int_{S^{d-1}}r^\beta P\left(v_n^{-1}{\|H\|}\in dr,\frac H{\|H\|}\in ds\right)\\
%&=&Cn\int_0^\vep r^\beta P\left(v_n^{-1}{\|H\|}\in dr\right)
&=&Cnv_n^{-\beta}\int_0^{\vep v_n}r^\beta P(\|H\|\in dr)
\to C\frac\alpha{\beta-\alpha}\vep^{\beta-\alpha}\sigma(S^{d-1})
\end{eqnarray*}
as $n\to\infty$, using \eqref{t3.eq7}, Karamata's Theorem, and \eqref{t1.eq5}. This shows that $\lim_{\vep\downarrow0}\limsup_{n\to\infty}\|I_{11}\|=0$. Finally, by similar calculations and the fact that $\|g(r,s)\|\le r$, it follows that
$$
\|I_{13}\|\le nv_n^{-1}\int_{Nv_n}^\infty rP(\|H\|\in dr)
\to\frac\alpha{\alpha-1}\sigma(S^{d-1})N^{1-\alpha}\,.
$$
This shows that $\lim_{N\to\infty}\limsup_{n\to\infty}\|I_{13}\|=0$. Thus, \eqref{t3.eq4} follows. By \eqref{t3.eq3} and \eqref{t3.eq4}, \eqref{t3.eq2} follows.

From Theorem \ref{t1} and \eqref{t3.eq2} it follows that
\begin{equation}\label{t3.claim.a}
v_n^{-1}S_n(n)-\frac n{v_n}E(H)=v_n^{-1}S_n(n)-a_n+a_n-\frac n{v_n}E(H)\Rightarrow\rho-\theta:=\rho_2
\end{equation}
so that $\rho_2$ has L\'evy representation $[-\theta,0,M]$.  Using \cite[Remark 3.1.15]{meerschaert:scheffler:2001}, we can write the log-characteristic function of a tempered stable law with mean zero in the form \eqref{LevyRepn2}.  Then it follows easily that \eqref{tsmean} holds.
\end{proof}

\begin{remark}
As noted in Section \ref{sec2}, we can center to zero expectation in \eqref{doadef} when $\alpha>1$, or dispense with the centering when $\alpha<1$.  Theorems \ref{t2} and \ref{t3} shows that the same centering can be used for the tempered random walk.  If $\alpha<1$, the limit is centered tempered stable, analogous to a centered stable law.  If $\alpha>1$, and we center to zero expectation for the {\em untempered} random walk jumps, the limit contains a shift depending on the spectral measure and the tempering function.  The shift comes from the fact that $I_1=nv_n^{-1}E\left[H_1-Y_{n1}\right]\to -m$ in \eqref{t3.eq4}.
\end{remark}

\begin{remark}
The special case $d=1$ is also important in applications (\cite{meerschaert:zhang:baeumer:2008}).  Suppose $d=1$, and that $H$ and $\pi(\cdot,\cdot)$ are as before. In this case, the conditional distribution of $T$ given $H$ can be written in a simpler form:
$P(T>u|H>0)=\pi(u,1)$ and $P(T>u|H<0)=\pi(u,-1)$.
Let $(v_n)$, $(a_n)$ and $\rho$ be as in Theorem \ref{t1}.
For $n\ge1$, suppose that $Y_{n1},\ldots,Y_{nn}$ are i.i.d. with
$$
Y_{n1}\eid\sgn(H)(|H|\wedge v_nT)\,.
$$
Let $S_n(k):=\sum_{j=1}^k Y_{nj}$. As a restatement of Theorem \ref{t1}, we obtain that
$$
v_n^{-1}S_n(n)-a_n\Rightarrow\rho\,.
$$
If $\alpha<1$, we can set $a_n=0$.  If $\alpha>1$, we can take $a_n=nv_n^{-1}E(H)$, provided
$$
\limsup_{u\downarrow0}\frac{2\alpha-q(u,1)-q(u,-1)}{u^{\beta-1}}<\infty
$$
for some $\beta>\alpha$.
\end{remark}

Let $\{X(t)\}$ be the L\'evy process generated by the tempered stable random vector $X$ with distribution $\rho$, so that $X(0)=0$ almost surely, $\{X(t)\}$ has stationary, independent increments, and $X(1)=X$ in distribution.  The next result shows that the tempered random walk \eqref{TRWdef} faithfully approximates the tempered stable process.

\begin{theorem}\label{t4}
Suppose that \eqref{t1.claim} holds as in Theorem \ref{t1}.  Then
\begin{equation}\label{t4.claim}
\{v_nS_n([nt])-ta_n\}_{t\geq 0}\Rightarrow \{X(t)\}_{t\geq 0}
\end{equation}
as $n\to\infty$ in the sense of finite dimensional distributions.
\end{theorem}

\begin{proof}
The L\'evy representation of the limit $\rho$ in \eqref{t1.claim} is $[0,0,M]$.  Use \eqref{t1.eq6} to get
$
\left[nt\right] P(v_n^{-1}Y_{n1}\in\cdot)\sim ntP(v_n^{-1}Y_{n1}\in\cdot)\vague tM(\cdot)
$
and \eqref{t1.eq4} to get
$
\lim_{\delta\downarrow0}\limsup_{n\to\infty}\left[nt\right] v_n^{-2}E\left\|Y_{n1}\one(\|Y_{n1}\|\le v_n\delta)\right\|^2=0\,.
$
Then $v_n^{-1}S_n\left(\left[nt\right]\right)-{a_n}t\Rightarrow\rho_t$ follows by the general convergence criteria for triangular arrays \cite[Theorem 3.2.2]{meerschaert:scheffler:2001}, where $\rho_t$ has L\'evy representation $[0,0,tM]$, since
\begin{eqnarray*}
\left\|{a_n}t-\left[nt\right]\int_{\{\|x\|<1\}}xP(v_n^{-1}Y_{n1}\in dx)\right\|
&\le&\int_{\{\|x\|<1\}}\|x\|P(v_n^{-1}Y_{n1}\in dx)\\
&\le&\left\{v_n^{-2}E\left\|Y_{n1}\one(\|Y_{n1}\|\le v_n)\right\|^2\right\}^{1/2}\to 0
\end{eqnarray*}
using \eqref{t1.eq7}.
To prove convergence of finite dimensional distributions, use the fact that increments of the random walk are independent.
\end{proof}

\begin{remark}
Take the exponential tempering function $q(r,s)=e^{-\lambda r}$ for $s=\pm 1$.  Then the random vectors $X(t)$ have smooth density functions $p(x,t)$ that solve a tempered fractional diffusion equation
$
{\partial_t}p=cq\,{\partial_{-x}^{\alpha,\lambda}} p+c(1-q)\,{\partial_{x}^{\alpha,\lambda}} p .
$
where $P(H<-r)/P(|H|>r)\sim q$ as $r\to\infty$.  The operator on the right hand side is the negative generator of the continuous convolution semigroup associated with $X$.   Some properties of the tempered fractional diffusion equation are developed in \cite{baeumer:meerschaert:2010}.  Theorem \ref{t4} shows that the tempered random walk \eqref{TRWdef} provides a useful approximation to the process $\{X(t)\}$.  In this case, the distribution of $T_i$ is given by
$P(T_i>u)=\pi(u,s)=u^\alpha\int_u^\infty r^{-\alpha-1}e^{-\lambda r}dr$,
which involves the incomplete gamma function.  The tempering thresholds $v_n$ do not depend on $q$.  For example, if $H$ belongs to the domain of normal attraction of some stable law, then we can take $v_n=cn^{1/\alpha}$ for some $c>0$.   Any random walk in the domain of attraction of a stable law can be modified using this tempering, to approximate an exponentially tempered stable.
\end{remark}

\begin{remark}
Suppose that the tempering variable is conditionally exponential with $P(T_i>t|\tfrac{H}{\|H\|}=s):=\pi(t,s)=e^{-\lambda_s t}$
for some continuous $s\mapsto \lambda_s>0$.
Let $h(r,s)=r^{-\alpha-1}q(r,s)$ and use \eqref{pidef} to get
$u^{-\alpha}e^{-\lambda_s u}=\int_u^\infty h(r,s)\,dr .$
Take derivatives with respect to $u$ on both sides to obtain
$-\alpha u^{-\alpha-1}e^{-\lambda_s u}-\lambda_s u^{-\alpha}e^{-\lambda_s u}=-h(u)$
and write
\begin{equation}\label{expT}
q(u,s)=u^{\alpha+1}h(u)=(\alpha+\lambda_s u)e^{-\lambda_s u} .
\end{equation}
Using this tempering function for the L\'evy measure \eqref{defM} yields a tempered stable law $X$ with a particularly simple tempering variable $T_i$.  If $1<\alpha<2$, then the form of the L\'evy measure shows that $X$ is the sum of two independent exponentially tempered stable laws, one with index $\alpha$, and the other with index $\alpha-1$.
\end{remark}

\begin{remark}
The goal of this paper is to construct random walk models that lead to a tempered stable limit.  To conclude this paper, we provide a practical, heuristic interpretation of those results.  A stable process serves to approximate a random walk with power-law jumps.  A tempered stable approximates the same random walk, once the largest jumps are reduced.  The tempering process represents an external force applied independently to each jump, the exact nature of which determines the tempered stable limit.   Any random walk in the domain of attraction of a stable, and subjected to this type of independent tempering, can be faithfully approximated by a tempered stable.  A few concrete examples are provided in \cite{meerschaert:roy:shao:2010}:  Precipitation data can be tempered due to atmospheric water content; measurements of hydraulic conductivity can be tempered by volume averaging; daily stock returns could be tempered by automatic trading limits.  See also \cite{aban:meerschaert:panorska:2006} for additional discussion.
\end{remark}

\section{Acknowledgement} The authors are grateful to an anonymous
referee for some comments that helped to improve the paper.

%\bibliography{D:/Mydata/work_arijit/res_ht/bibfile}

\begin{thebibliography}{}

\bibitem[Aban et~al., 2006]{aban:meerschaert:panorska:2006}
Aban, I., Meerschaert, M., and Panorska, A. (2006).
\newblock Parameter estimation for the truncated {P}areto distribution.
\newblock {\em Journal of the American Statistical Association},
  101(473):270--277.

%\bibitem[Araujo and Gin\'e, 1980]{araujo:gine:1980}
%Araujo, A. and Gin\'e, E. (1980).
%\newblock {\em The Central Limit Theorem for Real and {B}anach Valued Random
%  Variables}.
%\newblock Wiley, New York.

\bibitem[Baeumer and Meerschaert, 2010]{baeumer:meerschaert:2010}
Baeumer, B. and Meerschaert, M.~M. (2010).
\newblock Tempered stable L\'evy motion and transient super-diffusion.
\newblock {\em Journal of Computational and Applied Mathematics},
  233:243--2448.

\bibitem[Billingsley, 1968]{billingsley:1968}
Billingsley, P. (1968).
\newblock {\em Convergence of Probability Measures}.
\newblock Wiley, New York.

\bibitem[Carr et~al., 2002]{carr:geman:madan:yor:2002}
Carr, P., Geman, H., Madan, D.~B., and Yor, M. (2002).
\newblock The fine structure of asset returns: An empirical investigation.
\newblock {\em J. Business}, 75:303--325.

\bibitem[Carr et~al., 2003]{carr:geman:madan:yor:2003}
Carr, P., Geman, H., Madan, D.~B., and Yor, M. (2003).
\newblock Stochastic volatility for l\'evy processes.
\newblock {\em Math. Finance}, 13:345--382.

\bibitem[Koponen, 1995]{Koponen:1995}
Koponen, I. (1995).
\newblock Analytic approach to the problem of convergence of truncated l\'evy
  flights towards the gaussian stochastic process.
\newblock {\em Phys. Rev. E}, 52:1197--1199.

\bibitem[Meerschaert et~al., 1999]{meerschaert:benson:baeumer:1999}
Meerschaert, M.~M., Benson, D.~A., and Baeumer, B. (1999).
\newblock Multidimensional advection and fractional dispersion.
\newblock {\em Phys. Rev. E}, 59:5026--5028.

\bibitem[Meerschaert et~al., 2010]{meerschaert:roy:shao:2010}
Meerschaert, M.~M., Roy, P., and Shao, Q. (2010).
\newblock Parameter estimation for tempered power law distributions.
\newblock {\em Communications in Statistics -- Theory and Methods}, to appear.
Preprint available at \url{www.stt.msu.edu/~mcubed/TempPareto.pdf}.

\bibitem[Meerschaert and Scheffler, 2001]{meerschaert:scheffler:2001}
Meerschaert, M.~M. and Scheffler, H.-P. (2001).
\newblock {\em Limit Distributions for Sums of Independent Random Vectors:
  Heavy Tails in Theory and Practice}.
\newblock Wiley, New York.

\bibitem[Meerschaert and Scheffler, 2004]{meerschaert:scheffler:2004}
Meerschaert, M.~M. and Scheffler, H.-P. (2004).
\newblock Limit theorems for continuous time random walks with infinite mean
  waiting times.
\newblock {\em Journal of Applied Probability}, 41:623--638.

\bibitem[Meerschaert et~al., 2008]{meerschaert:zhang:baeumer:2008}
Meerschaert, M.~M., Zhang, Y., and Baeumer, B. (2008).
\newblock Tempered anomalous diffusions in heterogeneous systems.
\newblock {\em Geophysical Research Letters}, 35:L17403--L17407.

\bibitem[Novikov, 1994]{Novikov:1994}
Novikov, E.~A. (1994).
\newblock Infinitely divisible distributions in turbulence.
\newblock {\em Phys. Rev. E}, 50:R3303--R3305.

\bibitem[Resnick, 2007]{resnick:2007}
Resnick, S. (2007).
\newblock {\em Heavy-{T}ail {P}henomena : {P}robabilistic and {S}tatistical
  {M}odeling}.
\newblock Springer, New York.

\bibitem[Rosi\'nski, 2007]{Rosinski:2007}
Rosi\'nski, J. (2007).
\newblock Tempering stable processes.
\newblock {\em Stochastic Processes and their Applications}, 117(6):677 -- 707.

\bibitem[Samorodnitsky and Taqqu, 1994]{samorodnitsky:taqqu:1994}
Samorodnitsky, G. and Taqqu, M. (1994).
\newblock {\em {S}table {N}on-{G}aussian {R}andom {P}rocesses}.
\newblock Chapman and Hall, New York.

\bibitem[Sokolov et~al., 2004]{sokolov:chechkin:klafter:2004}
Sokolov, I.~M., Chechkin, A.~V., and Klafter, J. (2004).
\newblock Fractional diffusion equation for a power-law-truncated l\'evy
  process.
\newblock {\em Physica A}, 336:245--251.

\bibitem[van~der Vaart and Wellner, 1996]{vandervaart:wellner:1996}
van~der Vaart, A.~W. and Wellner, J.~A. (1996).
\newblock {\em Weak Convergence and Empirical Processes: With Applications to
  Statistics}.
\newblock Springer-Verlag, New York.

\bibitem[Zhang et~al., 2006]{zhang:benson:meerschaert:labolle:scheffler:2006}
Zhang, Y., Benson, D.~A., Meerschaert, M.~M., LaBolle, E.~M., and Scheffler,
  H.-P. (2006).
\newblock Random walk approximation of fractional-order multiscaling anomalous
  diffusion.
\newblock {\em Phys. Rev. E}, 74:026706.

\end{thebibliography}
%\bibliographystyle{D:/Mydata/work_arijit/res_ht/apalike}

\end{document}